\documentclass[a4paper]{article}
\usepackage{amsthm,amsfonts,amsmath,amssymb,units,bbold}
\usepackage[cp1251]{inputenc}
\usepackage[english]{babel}
\usepackage[final]{graphicx}
\usepackage{setspace}
\usepackage[12pt]{extsizes}
\begin{document}
\newcommand{\diag}{{\rm diag}}
\newtheorem{teorema}{Theorem}
\newtheorem{lemma}{Lemma}
\newtheorem{utv}{Proposition}
\newtheorem{svoistvo}{Property}
\newtheorem{sled}{Corollary}
\newtheorem{con}{Conjecture}
\newtheorem{zam}{Remark}

\author{A.A. Taranenko\footnote{Sobolev Institute of Mathematics, taa@math.nsc.ru}}
\title{Algebraic properties of perfect structures}
\date{April 19, 2020}
\maketitle
\begin{abstract}
A perfect structure is a triple $(M,P,S)$ of matrices $M, P$ and $S$ of consistent sizes such that $MP = PS$. Perfect structures comprise similar matrices, eigenvectors, perfect colorings (equitable partitions) and graph coverings. In this paper we study general algebraic properties of perfect structures and characterize all perfect structures with identity or unity matrix $M$. Next, we consider a graph product generalizing most standard products  (e.g. Cartesian, tensor, normal, lexicographic graph products). For this product we propose a construction of perfect structures and prove that it can be reversed for eigenvectors. Finally, we apply obtained results to calculate the spectra of several classes of graphs and to prove some properties of perfect colorings.
\end{abstract}

\medskip


\section*{Introduction}

The main aim of this paper is to introduce and expose the notion of a perfect structure. This notion binds together such basic linear algebraic objects as eigenvectors and similar matrices with a vast enough class of combinatorial structures, namely with perfect colorings and graph coverings. 

At times, perfect colorings arose in the literature under different names. One of the most known of them is ``equitable partition'' that was introduced by Delsarte~\cite{delsart} while he studied completely regular codes. In book~\cite{doob} objects equivalent to perfect colorings are named as divisors of graphs. An algebraic definition of perfect colorings firstly appeared in book~\cite{godalgcomb}. At last, in book~\cite{godsil} one finds some information on graph coverings that can be considered as perfect colorings of a special type. Using an algebraic approach similar to the present one, the same definition of perfect structures and some of their properties arose earlier in paper~\cite{krotov}.

Strictly speaking, all the above books and papers have already contained most properties of perfect structures or graph products that will be in the center of attention of this paper. Meanwhile, in this study we try to cover an as wide as possible range of basic results for perfect structures treating them under consolidated terminology. 

Let us describe the structure of the paper.
Section~\ref{intersec} is preliminary and aims to recall basic notions of algebra and graph theory, including the most popular graph products. In Section~\ref{perfsec} we introduce the notion of perfect structures and obtain their main algebraic properties. In particular, in Subsection~\ref{structIJ} we completely describe two basic classes of perfect structures used as building bricks for more complicated structures: perfect structures with the identity and unity adjacency matrices. 

In Section~\ref{prodsec} we study perfect structures in graph products. We introduce a generalized graph product so that the most common products (such as the tensor, Cartesian, normal and lexicographic products of graphs~\cite{handprod}) will be special cases of our generalized product. In  Subsection~\ref{prodpar}  we propose a general construction of the product of perfect structures in the introduced graph product and specify it for different graph products and perfect structures. 
 
 In Subsection~\ref{projectpar} we prove a contraction theorem for eigenfunctions that reverses the structure product construction. A special case of this theorem was the main tool for the bound on the minimal support of eigenfunctions in Hamming graphs~\cite{val}, and later the product constructions of eigenfunctions in these graphs were used in~\cite{valvor}. 
 
 Section~\ref{primsec} is devoted to applications of previous results to certain classes of graphs and perfect colorings. The spectral graph properties are extensively studied in the algebraic graph theory (see, for example,~\cite{brower}).  So in Subsection~\ref{spectrsection} we apply the graph product method to calculating the spectra of graphs. Some corollaries of our results for perfect colorings are obtained in Subsection~\ref{perfcolpar}.
 
Techniques and results of this paper can be also applied for the characterization problem of all parameter matrices of perfect colorings in a given graph. For example, in~\cite{flaass} this problem was studied for perfect $2$-colorings of the Hamming graphs $H(n,2)$ (colorings of the Boolean $n$-hypercubes). Next, all perfect colorings of the prism graphs were described in~\cite{lis}. At last, in~\cite{perfz2} there are lists of all perfect colorings up to $9$ colors of the infinite square grid (i.e., the Cartesian product of two infinite chains) and their parameter matrices.

\section{Notions and definitions} \label{intersec}

\subsection{Graphs, adjacency matrices, and spectra}

Given a directed multigraph $G = (V,E)$ with the vertex set $V$, $|V| = n$, and the arc set $E$, the \textit{adjacency matrix} $M = M(G)$ of $G$ is an $n \times n$ matrix with entries $m_{i,j}$ equal to the number of arcs from vertex $i$ to vertex $j$. The adjacency matrix of a simple undirected graph is exactly a symmetric $(0,1)$-matrix  with zero entries within the main diagonal. The \textit{complete graph} has the adjacency matrix with all entries, except diagonal, equal to one.

For every complex $n \times n$ matrix $M$ we can assign a digraph (allowing loops) on $n$ vertices with arcs labeled by entries of the matrix $M$. So every complex matrix can be treated as the adjacency matrix of some ``graph''. Meanwhile, in order to preserve certain algebraic properties of adjacency matrices of graphs, we focus only on diagonalizable matrices. Recall that a complex square matrix is called \textit{diagonalizable} if it is similar to a  diagonal matrix (has a diagonal Jordan normal form). It is well known that every real symmetric matrix is diagonalizable and that a complex $n\times n$ matrix $M$ is diagonalizable if and only if there is a basis of the space $\mathbb{C}^n$ composed by eigenvectors. 

In the rest of the paper we deal with only complex adjacency matrices but, under some natural assumptions, most of the future results remain true for adjacency matrices with entries from other fields. 

We will say that a graph $G$ is \textit{regular} if its adjacency matrix is symmetric and has equal row and column sums which are called the \textit{degree} of the regular graph.

The \textit{spectrum} of a graph $G$ is the spectrum of its adjacency matrix $M(G)$, i.e., the multiset of the eigenvalues of $M(G)$.

In what follows, $I$ denotes the identity matrix having unity entries on the main diagonal and zeroes elsewhere, $J$ is used for the unity matrix whose all entries are equal to one, and $\diag(\lambda_1, \ldots, \lambda_n)$ is a diagonal matrix with $ \lambda_1, \ldots, \lambda_n$ on the main diagonal. We often do not specify orders of matrices when they are clear from the surrounding.

\subsection{Kronecker product of matrices}

Given matrices $M = (m_{i,j})_{i,j=1}^n$ and $L$ of orders $n$ and $r$ respectively, the \textit{Kronecker product} $M \otimes L$ is the following block matrix of order $nr$: 

$$M \otimes L = \left(\begin{array}{ccc} m_{1,1} L & \cdots & m_{1,n} L \\ \vdots & \ddots & \vdots  \\ m_{n,1} L & \cdots & m_{n,n} L \end{array} \right).$$

There are the main properties of the Kronecker product that will be used in this paper:

\begin{enumerate}
\item $M \otimes (L \otimes S) = (M \otimes L) \otimes S$.
\item $M \otimes (L + S) = M \otimes L + M \otimes S$ and $(M + L) \otimes S = M \otimes S + L \otimes S$.
\item $\alpha (M \otimes L) = (\alpha M) \otimes L = M \otimes (\alpha L)$  for each $\alpha \in \mathbb{C}$.
\item If products $MS$ and $L R$ exist, then $(M \otimes L)(S \otimes R) = M S \otimes L R$.
\end{enumerate}

\subsection{Graph products}

Following the standard definitions, a \textit{product} of graphs $G = (V, E)$ and $H = (U, W)$ is a graph with a vertex set $V \times U$ such that the adjacency of vertices $(v_1, u_1)$ and $(v_2,u_2)$ is  defined on the base of  the edge sets  $E$ and $W$. Specializing the adjacency of vertices, we obtain a variety of graph products.

Let us define the most common graph products. In these definitions we assume that graphs $G = (V, E)$ and $H = (U, W)$  have adjacency matrices $M$ and $L$ respectively.

The \textit{tensor product $G \times H$} is a graph such that its vertices  $(v_1, u_1)$ and $(v_2,u_2)$ are adjacent if  and only if $v_1 v_2 \in E$ and $u_1 u_2 \in W$.   
The adjacency matrix of the tensor product $G \times H$ is 
$$M(G \times H) = M \otimes L.$$

The \textit{Cartesian product $G \Box H$} is a graph such that its vertices  $(v_1, u_1)$ and $(v_2,u_2)$ are adjacent if and only if $v_1 = v_2$ and $u_1 u_2 \in W$ or $u_1 = u_2$ and $v_1 v_2 \in E$.
The adjacency matrix of the Cartesian product $G \Box H$ is 
$$M(G \Box H) = M \otimes I + I \otimes L.$$

The \textit{normal product $G \boxtimes H$} is a graph such that its vertices  $(v_1, u_1)$ and $(v_2,u_2)$ are adjacent if and only if $v_1 = v_2$ and $u_1 u_2 \in W$, or $u_1 = u_2$ and $v_1 v_2 \in E$ or $v_1 v_2 \in E$ and $u_1 u_2 \in W$. 
The adjacency matrix of the normal product $G \boxtimes H$ is 
$$M(G \boxtimes H) = M \otimes I + I \otimes L + M \otimes L.$$

The \textit{lexicographic product $G \cdot H$} is a graph such that its vertices  $(v_1, u_1)$ and $(v_2,u_2)$ are adjacent if and only if $v_1 v_2 \in E$ or $v_1 = v_2$ and $u_1 u_2 \in W$.   
The adjacency matrix of the lexicographic product $G \cdot H$ is 
$$M(G \cdot H) = M \otimes J + I \otimes L.$$

\section{Basics on perfect structures} \label{perfsec}

A \textit{perfect structure} is a triple of complex matrices $(M,P,S)$ satisfying the equation
$$MP = PS.$$

Given this,
\begin{itemize}
\item the square matrix $M$ of order $n$ is diagonalizable and is the \textit{adjacency matrix} of a graph;
\item the rectangular $n \times k$ matrix $P$ with $k \leq n$ is called the \textit{structure matrix};
\item the square matrix $S$ of order $k$ is called the \textit{parameter matrix}.
\end{itemize}

Equivalently, we will say that $P$ is the perfect structure with parameters $S$ in a graph defined by the adjacency matrix $M$.

Some of our future results require an additional condition on the structure matrix $P$.  We define a class of \textit{nonsingular} perfect structures to be the set of all perfect structures with the structure matrix $P$ of full rank.

Let us consider the key examples of perfect structures.

Our first example is pairs of similar matrices. As is well known, matrices $A$ and $B$ are said to be \textit{similar} if there exists a nonsingular matrix $C$ such that $A = CBC^{-1}$ which is equivalent to $AC = CB$. Thus we have that similar matrices $A$ and $B$ are the nonsingular perfect structure $(A,C,B)$ with the structure matrix $C$.

Another example of perfect structures is eigenvectors of matrices (or eigenfunctions of graphs). Recall that an \textit{eigenvector} of a complex matrix $M$ is a nonzero vector $f$ such that $Mf = \lambda f$, where $\lambda \in \mathbb{C}$ is called by an \textit{eigenvalue}. An \textit{eigenfunction} of a graph $G = (V,E)$ is a function $f: V \rightarrow \mathbb{C}$ that is an eigenvector of the adjacency matrix of $G$. It is easy to note that every eigenvector (or eigenfunction) $f$ corresponds to a perfect structure $(M,f,\lambda)$ with a scalar parameter matrix $\lambda$.

Let us turn to combinatorial examples of perfect structures.

 A \textit{perfect coloring in $k$ colors} of a graph $G = (V,E)$ on $n$ vertices is a surjective function $c: V \rightarrow \{ 1, \ldots, k\}$ such that if $c(u) = c(v)$ for some vertices $u,v \in V$, then multisets $\{ c(w)| uw \in E(G)\}$ and  $\{ c(w)| vw \in E(G)\}$ coincide. Given a perfect coloring $c$, let the $(0,1)$-matrix $P$ be an $n \times k$ matrix with entries $p_{u,j} = 1$ whenever $c(u) = j$. It is not hard to see that  the equality $MP = PS$ holds for the adjacency matrix $M$ of the graph $G$ and for an appropriate integer matrix $S$ called by the \textit{parameters of a perfect coloring}. More specifically, entries $s_{i,j}$ of the parameter matrix $S$ are equal to  the number of vertices of color $j$ in the neighborhood of any vertex of color $i$.
The set of perfect colorings of a graph $G$ with the adjacency matrix $M$ is exactly the set of all  nonsingular perfect structures $(M,P,S)$, where the structure matrix $P$ is a $(0,1)$-matrix with exactly one unity entry in each row.

Every perfect coloring in $k$ colors of a simple graph $G$ defines a partition of the vertex set of $G$ into $k$ parts such that the induced subgraph on the vertices of each part is regular and the edges of $G$ between vertices from different classes compose a biregular bipartite graph. In other words, each color class of a perfect coloring is a part of an equitable partition, and the parameter matrix of a perfect coloring is the quotient matrix of the partition.

From the point of view of  adjacency matrices, the existence of a perfect coloring  means that after an appropriate row and column permutation the adjacency matrix $M$ can be written in a block-diagonal form $M = \{B_{i,j}\}_{i,j=1}^k$, where $B_{i,j} = B_{j,i}^T$ are rectangular  blocks with equal row and column sums.

Let us introduce graph coverings that are another important special case of the perfect structures. A graph $G = (V,E)$ is said to \textit{cover} a graph $H = (U,W)$ if there exists a surjective function $\varphi : V \rightarrow U$ such that for each $v \in V$ the equality $\{\varphi(u)| (v,u) \in E  \} = \{ w | (w,\varphi(v)) \in W\}$ holds. It is not hard to see that a graph $G$ covers a graph $H$ if and only if there is a perfect coloring $(M,P,S)$ in which $M$ is the adjacency matrix of the graph $G$ and $S$ is the adjacency matrix of the graph $H$. 

At last, we mention one more natural generalization of perfect colorings. Let a \textit{fractional perfect coloring} of a graph $G$ with the adjacency matrix $M$ be a nonsingular perfect structure $(M,P,S)$ such that all entries of the matrix $P$ are nonnegative real numbers and all row sums of $P$ are equal to $1$. Every convex combination of perfect colorings in a graph with the same parameter matrix $S$ is a fractional perfect coloring. 

\subsection{Main properties of perfect structures} \label{perfprop}

\begin{svoistvo}[Linearity of the space of structure matrices]\label{sructspace}
\item Given the adjacency matrix $M$ and the parameter matrix $S$ of a perfect structure $(M,P,S)$, the set of the structure matrices $P$ is a linear space.
\end{svoistvo}

\begin{proof}
If  $(M, P, S)$ and  $(M, R, S)$ are perfect structures, then for any complex $\alpha$ and $\beta$ the triple $(M, \alpha P + \beta R, S)$ is a perfect structure, because 
$$M (\alpha P + \beta R) = \alpha MP + \beta MR = \alpha P S + \beta R S = (\alpha P + \beta R) S.$$
\end{proof}

\begin{svoistvo}[Synergies between adjacency and parameter matrices] \label{structgrsum}
\item
\begin{enumerate}
\item If $(M,P,S)$  is a perfect structure, then for all $\alpha \in \mathbb{C}$ the triple $(\alpha M, P, \alpha S)$ is a perfect structure.
\item If $(M, P, S)$ and $(L, P, T)$ are perfect structures and $M + L$ is a diagonalizable matrix, then $(M +  L, P,  S + T)$ is a perfect structure.
\item If  $(M,P,S)$ is a perfect structure, then for all  $k \in \mathbb{N}$ the triple $(M^k, P, S^k)$ is a perfect structure.
\item Given a polynomial $p(x) \in \mathbb{C}[x]$ and a perfect structure $(M,P,S)$, the triple $(p(M), P, p(S))$ is a perfect structure.
\end{enumerate}
\end{svoistvo}
\begin{proof}
\begin{enumerate}
\item $(\alpha M)P = \alpha (MP) = \alpha (PS) = P(\alpha S).$
\item $(M+ L)P = M P + L P = PS + PT = P(S + T).$
\item $M^{k} P = M^{k-1}( MP) =  M^{k-1} (PS) =\ldots = P S^{k-1} S = PS^{k}.$
\item This clause easily follows from the previous ones. 
\end{enumerate}
\end{proof}

\begin{svoistvo}[Kronecker product of perfect structures] \label{structtensorprod}
\item If triples $(M, P, S)$ and $(L, R, T)$ are perfect structures, then the triple $(M \otimes L, P \otimes R, S \otimes T)$ is a perfect structure.
\end{svoistvo}

\begin{proof}
Using the properties of the Kronecker product, we have
$$ (M \otimes L) (P \otimes R) = MP \otimes LR =  
PS \otimes R T = (P \otimes R)(S \otimes T). $$

\end{proof}

\begin{svoistvo}[Composition of perfect structures] \label{structcomp}
\item If $(M,P, S)$ and $(S, R, T)$ are (nonsingular) perfect structures, then the triple $(M, PR, T)$ is a (nonsingular) perfect structure.
\end{svoistvo}

\begin{proof}
Equalities $M P = P S$ and $S R = R T$ imply that 
$$M PR = P S R = P R T.$$
If perfect structures $(M,P, S)$ and $(S, R, T)$ are nonsingular, then the perfect structure $(M,PR,T)$ is nonsingular because of the inequality $rank(P R) \geq rank(P) + rank(R) - k$, where $k$ is the number of columns of $P$.
\end{proof}

Taking the composition of a perfect structure with an eigenvector of the parameter matrix, we obtain the following property.

\begin{svoistvo}[Mapping of eigenvectors] \label{eigenvlog}
\item If $f$ is an eigenvector corresponding to an eigenvalue $\lambda$ for the parameter matrix $S$ of a nonsingular perfect structure $(M,P,S)$, then $Pf$ is an eigenvector of the adjacency matrix  $M$  for the same eigenvalue  $\lambda$.
\end{svoistvo}

\begin{svoistvo}[Similarity of perfect structures] \label{conjstruct}
\item If $(M,P,S)$ is a (nonsingular) perfect structure, then for every nonsingular matrix $A$ of order $n$  and every nonsingular matrix $B$ of order $k$,  the triple $(M', P', S')$, where  $M' =A M A^{-1}$, $P' = A P B^{-1}$ and $S' = B S B^{-1}$,  is a (nonsingular) perfect structure.
\end{svoistvo}

\begin{proof}
Indeed,
$$(A M A^{-1})(A P B^{-1}) = A M P B^{-1} = A P S B^{-1}  = (A P B^{-1})(B S B^{-1}).$$
If the structure $(M,P,S)$ is nonsingular, then the matrix $P'$ has a full rank because matrices $A$ and $B^{-1}$ are nonsingular and $P$ has a  full rank.
\end{proof} 

We will say that perfect structures $(M,P,S)$ and $(M',P',S')$ are \textit{similar} if for some nonsingular matrices $A$ and $B$ we have $M' =A M A^{-1}$, $P' = A P B^{-1}$ and $S' = B S B^{-1}$.

\begin{svoistvo}[Jordan normal form of the parameter matrix] \label{symspecS}
\item If the triple $(M,P,S)$ is a nonsingular perfect structure then the following hold.
\begin{itemize}
\item  The parameter matrix $S$ is diagonalizable.
\item The spectrum of the parameter matrix $S$ is included in the spectrum of the adjacency matrix $M$ as a multiset.
\end{itemize}
\end{svoistvo}

\begin{proof}
Assume that the parameter matrix $S$ is not diagonalizable. Then there is a generalized eigenvector $f$ for some eigenvalue $\lambda$ such that $(S - \lambda I) f \neq 0 $ but $(S - \lambda I)^2 f = 0$.

Since the matrix $P$ has a full rank, we have $Pf \neq 0$.  Using Property~\ref{structgrsum}, we obtain
$$(M - \lambda I) Pf = P(S - \lambda I) f \neq 0, $$
because $(S -\lambda I )f \neq 0 $ and $P$ has full rank. On the other hand,
$$(M - \lambda I)^2 Pf = P(S - \lambda I)^2 f =  0,$$
since $(S- \lambda I)^2 f = 0$. Therefore, $Pf$ is a generalized eigenvector for the matrix $M$ that contradicts to the diagonalizability of this matrix. 

Let us prove now that the spectrum of the matrix $M$ includes the spectrum of $S$. Since the matrix $S$ is diagonalizable, there is a set $f_1, \ldots, f_k$ of its linearly independent eigenvectors corresponding to eigenvalues $\mu_1, \ldots, \mu_k$. Then the vectors $Pf_1, \ldots, Pf_k$ are also linearly independent because the matrix $P$ has a full rank. Property~\ref{eigenvlog} implies that vectors $Pf_1, \ldots, Pf_k$ are eigenvectors for the matrix $M$ corresponding to eigenvalues $\mu_1, \ldots, \mu_k$.
\end{proof}

Property~\ref{symspecS} allows us to put every perfect structure to a canonical form. The \textit{canonical form} of a perfect structure $(M, P, S)$ is a similar perfect structure $(M,R,T)$ with a diagonal parameter matrix $T$. The canonical form of a perfect structure is unique up to concurrent row and column permutations.

\begin{svoistvo}[Canonical form of a perfect structure] \label{canonstruct}
\item Every nonsingular perfect structure $(M,P,S)$ is congruent to the perfect structure $(M, R, T)$, where the matrix $T = \diag(\mu_1, \ldots, \mu_k)$. Moreover, the multiset $\{\mu_1, \ldots, \mu_k \}$ is the spectrum of the parameter matrix $S$, and columns $R_i$ of the matrix $R$  are the linearly independent eigenvectors of the matrix  $M$ such that  $MR_i = \mu_i R_i$.
\end{svoistvo}

\begin{proof}
By Properties~\ref{conjstruct} and~\ref{symspecS}, we have that the required perfect structure $(M,R,T)$ and the matrix $T$ exist. It only remains to note that the equality $MR =  RT$  is equivalent to equalities $MR_i = \mu_i R_i$ for all $i = 1, \ldots, k$.
\end{proof}

\begin{svoistvo}[Columns of the structure matrix] \label{columnspaceP}
\item If $(M,P,S)$ is a nonsingular perfect structure, then each column $P_i$ of the structure matrix $P$ belongs to the sum of eigenspaces of the adjacency  matrix $M$ coresponding to the eigenvalues of the parameter matrix  $S$.
\end{svoistvo}

\begin{proof}
By Property~\ref{canonstruct}, there is the canonical form $(M,R,T)$ of the perfect structure $(M,P,S)$  such that the columns of $R$ are the eigenvectors of $M$ corresponding to the eigenvalues of the parameter matrix $S$. It only remains to note that there is some nonsingular matrix $B$ for which $P = RB$, because the perfect structures $(M,R,T)$ and $(M,P,S)$ are similar.
\end{proof}

\begin{svoistvo}[Orthogonality of structure matrices]
\item If $(M,P,S)$ and $(M,R,T)$ are nonsingular perfect structures in a graph with a symmetric adjacency matrix $M$ and the spectra of the parameter matrices $S$ and $T$ do not intersect, then each column $P_i$ of the structure matrix $P$ is orthogonal to any column $R_j$ of the structure matrix $R$. 
\end{svoistvo}

\begin{proof}
By Property~\ref{columnspaceP}, columns $P_i$ and $R_j$ belong to sums of eigenspaces corresponding to nonintersected subsets of eigenvalues of $M$. Since $M$ is a symmetric matrix,  eigenspaces corresponding to different eigenvalues are orthogonal. Thus the invariant subspaces of $M$ containing vectors $P_i$ and $R_j$ are also orthogonal. 
\end{proof}

\begin{svoistvo}[Dimension of the space of structure matrices]
\item
\begin{enumerate}
 \item Given diagonalizable matrices $M$ and $S$, there exists a nonsingular perfect structure $(M,P,S)$ with a structure matrix $P$ of a full rank if and only if the spectrum of the matrix $M$ includes the spectrum of the matrix $S$ as a multiset. 
 \item The linear dimension of the space of structure matrices $P$ for given diagonalizable matrices $M$ and $S$ is equal to $\sum\limits_{i} \nu_M (\lambda_i) \nu_S (\lambda_i)$, where $\nu_M(\lambda_i)$ is the multiplicity of the eigenvalue $\lambda_i$ for the adjacency matrix  $M$ and $\nu_S (\lambda_i)$  is the multiplicity of the same eigenvalue in the parameter matrix $S$.
 \end{enumerate}
\end{svoistvo}

\begin{proof}
1. A proof of the fact that inclusion of the spectrum of $S$ in the spectrum of the matrix $M$ is necessary for the existence of a nonsingular perfect structure $(M,P,S)$ was given in Property~\ref{symspecS}. To prove sufficiency, we put $P = RB$, where $B$ is a nonsingular matrix conjugating the matrix $S$ to the diagonal form, and $R$ is a rectangular matrix whose columns compose a set of linearly independent eigenvectors corresponding to the eigenvalues of $S$.   Such a matrix $R$ exists because the spectrum of $S$ is contained in the spectrum of $M$.

2. For a perfect structure $(M,P,S)$, let us consider the similar perfect structure $(J_M, R, J_S)$, where $J_M$ and $J_S$ are diagonal Jordan normal forms of the matrices $M$ and $S$ respectively. The dimension of the space of perfect structures $(M,P,S)$ coincides with the dimension of the space of structures $(J_M, R, J_S)$. The equation $J_M R = R J_S$ is equivalent to equalities $\lambda_i r_{i,j} = \mu_j r_{i,j}$ for all $i =1, \ldots, n$ and $j= 1, \ldots, k$. So entries $r_{i,j} = 0$ if $\lambda_i \neq \mu_j$ and $r_{i,j}$ is arbitrary otherwise. Therefore, the dimension of the space of structure matrices $R$ is equal to $\sum\limits_{i} \nu_M (\lambda_i) \nu_S (\lambda_i)$.
\end{proof}

\begin{svoistvo}[Existence and uniqueness of parameter matrices]
\item Given an adjacency matrix $M$ of order $n$ and an $n \times k$ structure matrix $P$, there exists a nonsingular perfect structure $(M,P,S)$ with the parameter matrix $S$ if and only if the linear span of the columns of $P$ is an invariant subspace of dimension $k$ for the matrix $M$. Moreover, under these conditions the parameter matrix $S$ and the perfect structure $(M,P,S)$ are unique.
\end{svoistvo}

\begin{proof}
We start with the proof of necessity. If the dimension of the linear span of columns of $P$ is less than $k$, then the matrix $P$ is not of full rank and, therefore, the structure $(M,P,S)$ cannot be nonsingular. Equality $MP = PS$ considered for columns means that the action of the matrix $M$ on any column of the matrix $P$ gives a vector from the linear span of its columns.

Let us prove sufficiency now. If the linear span of the columns of $P$ is an invariant subspace of the matrix $M$ of dimension $k$, then there exists a base set of eigenvectors $R_1, \ldots, R_k$ of $M$ corresponding to eigenvalues $\mu_1, \ldots, \mu_k$ such that  $P = RB$ for some nonsingular matrix $B$ and the matrix $R$ having $R_1, \ldots, R_k$ as columns. Then we put $S = B^{-1} \diag\{\mu_1, \ldots, \mu_k\}$.

If for given matrices $M$ and $P$ the parameter matrix $S$ is not unique, then there exists a nonzero matrix $T$ such that $PT = 0$. This equation corresponds to a set of homogeneous systems of linear equations on the columns of $P$ and has only zero solution if the matrix $P$ is of full rank. 
\end{proof}

\subsection{Perfect structures with adjacency matrices $I$ and $J$} \label{structIJ}

\begin{utv}[Identity adjacency matrix] \label{structI}
\item All (nonsingular) perfect structures in a graph with adjacency matrix $I$ are structures  $(I,P,I)$, where $P$ is an arbitrary $n \times k$ matrix (of full rank).
\end{utv}

\begin{proof}
Properties~\ref{symspecS} and~\ref{canonstruct} imply that the parameter  matrix $S$ of a perfect structure $(I,P,S)$ is similar to the identity matrix and so coincides with it.
Equality $IP = PI$ trivially holds for every matrix $P$. 
\end{proof}

From Propostion~\ref{structI} and Property~\ref{structgrsum}, we easily get the following property.

\begin{sled} \label{sdvignaI}
A triple $(M,P,S)$ is a perfect structure if and only if for any $\alpha \in \mathbb{C}$ the triple $(M + \alpha I, P , S + \alpha I)$ is a perfect structure.
\end{sled}

Let us characterize all perfect structures in graphs with the unity adjacency matrix now.

\begin{utv}[Unity adjacency matrix] \label{completegrstruct}
\item A triple $(J,P,S) $ with the  matrix $J$ of order $n$ and the matrix $S$ of order $k$ is a nonsingular perfect structure if and only if one of the following possibilities holds.
\begin{itemize}
\item $S = \mathbb{0}$ (the zero matrix) and $P$ is an arbitrary full rank matrix with zero column sums;
\item  $S$ is a matrix of rank $1$, entries $s_{i,j}$ of $S$ are equal to $n v_i u_j$ for some $v_i, u_j \in \mathbb{C}$, $i ,j= 1, \ldots, k$.  $P$ is a full rank matrix with the sum of entries in the $j$th column equal to $n u_j \sum\limits_{t=1}^k p_{i,t}v_t$ for all $i = 1, \ldots, n$.
\end{itemize}
\end{utv}

\begin{proof}
The spectrum of the matrix $J$ of order $n$ consists of eigenvalue $n$ of multiplicity $1$ and eigenvalue $0$ of multiplicity $n-1$. By Property~\ref{symspecS}, the spectrum of the matrix $S$ is a subset of the spectrum of $M$, therefore it either contains the only eigenvalue $0$ of multiplicity $k$ or consists of eigenvalue $n$ of multiplicity $1$ and eigenvalue $0$ of multiplicity $k-1$. The first case gives us the matrix $S$ composed by zeroes, and the equality $JP = P \mathbb{0}$ means that all columns sums of $P$ are zero.
 
 If the spectrum of the matrix $S$ contains the eigenvalue $n$, then entries $s_{i,j}$ of the matrix $S$ can be presented as $n v_i u_j$ for appropriate complex numbers $v_i$ and $u_j$. The equality $JP = PS$ gives that the sum of entries in $j$th column of  every structure matrix $P$ is $n u_j \sum\limits_{t=1}^k p_{i,t}v_t$ for all $i = 1, \ldots, n$.
\end{proof}

Let us obtain some corollaries of Propositions~\ref{structI} and~\ref{completegrstruct}.
Firstly, note that the adjacency matrix of the complete graph $K_n$ differs from the matrix $J$ by the identity matrix $I$. Therefore, Corollary~\ref{sdvignaI} allows us not to distinguish these two matrices and to describe all perfect colorings of complete graphs.

\begin{sled} \label{colorinJ}
Let  $P$ correspond to an arbitrary coloring in $k$ colors of vertices of the complete graph  and assume that the number of vertices of color $i$ equals $n_i$. Then $P$ is a perfect coloring of the complete graph with the parameter matrix $J \cdot \diag(n_1, \ldots, n_k) - I$.
\end{sled}

\begin{proof}
It can be verified directly that the triple $(J, P, J  \cdot \diag(n_1, \ldots, n_k)) $ is a perfect structure.
\end{proof}

The following corollary binds eigenfunctions in a regular graph and eigenvectors of the matrix $J$.

\begin{sled} \label{commonJ}
Every eigenfunction $f$ in a regular connected graph on $n$ vertices is also an eigenfunction in the complete graph on the same vertices and is an eigenvector of the matrix $J$. If $f$ is collinear to the vector $\mathbb{1}$ with all unity entries, then $f$ is the eigenfunction of $J$ corresponding to the eigenvalue $n$, otherwise it corresponds to the eigenvalue $0$.
\end{sled}

\begin{proof}
The vector $\mathbb{1}$  is an eigenvector as for the adjacency matrix of a regular graph corresponding to eigenvalue equal to the degree of the graph, as for the matrix $J$, where it corresponds to the eigenvalue $n$.  Because the adjacency matrices of connected graphs are symmetric and irreducible, all other eigenvectors $g$ are orthogonal to the vector $\mathbb{1}$ and, consequently, the sum of all entries of $g$ is $0$. Therefore, vectors $g$ are the eigenvectors of the matrix $J$ corresponding to the eigenvalue $0$.
\end{proof}

One can use the above results, for example, for finding the spectrum of the complement of a regular graph.

\begin{utv}
If $M$ is the adjacency matrix of a regular graph $G$ on $n$ vertices and of degree $r$ and $f$ is an eigenfunction of $G$, then $f$ is an eigenfunction in the complement graph $\overline{G}$ having the adjacency matrix $J - M - I$. The spectrum of the complement graph is
$$sp(\overline{G}) = \left\{ -\lambda -1| \lambda \in sp(G), \lambda \neq r  \right\} \cup  \left\{ n-r - 1 \right\}.$$
\end{utv}

\begin{proof}
By Corollary~\ref{commonJ}, each eigenfunction $f$ of a regular graph is an eigenvector of the matrix $J$ corresponding to either the eigenvalue  $n$ (when $f$ is collinear to $\mathbb{1}$) or eigenvalue $0$ (otherwise). By Proposition~\ref{structI}, every vector is an eigenvector of the matrix $I$ with eigenvalue $1$. To complete the proof, we apply Property~\ref{structgrsum}.
\end{proof}

\section{Perfect structures in graph products} \label{prodsec}

We start the section with the following generalization of the graph product.
Given graphs with adjacency matrices $M_1, \ldots, M_m$ of the same order and graphs with adjacency matrices $L_1, \ldots, L_l$ of some other order, we define their  \textit{$( \alpha_{i,j} )$-product} to be the graph with the adjacency matrix
$$\sum\limits_{i=1}^m \sum\limits_{j=1}^l \alpha_{i,j} M_{i} \otimes L_{j}.$$

\subsection{Product of perfect structures} \label{prodpar}

The main result of this subsection is the following theorem on the product of perfect structures in the $(\alpha_{i,j})$-product of graphs. The theorem is an easy corollary of Properties~\ref{structgrsum} and~\ref{structtensorprod}.

\begin{teorema} \label{structgenprod}
For all  $\alpha_{i,j} \in R$ and collections of perfect structures $(M_1, P, S_1)$, $\ldots$ , $(M_{m}, P, S_{m})$ and $(L_1, R, T_1), \ldots , (L_{l}, R, T_{l})$ such that structure matrices $P$ and $R$ are the same within a collection, the triple of matrices
$$\left( \sum\limits_{i=1}^m \sum\limits_{j=1}^l \alpha_{i,j} M_{i} \otimes L_{j},~ P \otimes R,~\sum\limits_{i=1}^m \sum\limits_{j=1}^l\alpha_{i,j} S_{i} \otimes T_{j} \right)$$
is a perfect structure.
\end{teorema}

In other words, Theorem~\ref{structgenprod} says that if $P$ and $R$ are perfect structures in graphs with adjacency matrices $M_1, \ldots, M_m$ and $L_1, \ldots, L_l$ respectively, then the tensor product $P \otimes R$ is a perfect structure in the $(\alpha_{i,j})$-product of these graphs with the parameter matrix equal to the $(\alpha_{i,j})$-product of parameter matrices.

It is not hard to specialize Theorem~\ref{structgenprod} for the tensor, Cartesian and normal products of graphs. In what follows we assume that $(M,P,S)$ and $(L,R,T)$ are perfect structures in graphs $G$ and $H$ with adjacency matrices $M$ and $L$ respectively and that all left (and right) terms in Kronecker products have the same order.

\begin{teorema} \label{structtensprod}
The matrix $P \otimes R$ is a perfect structure in the tensor product of graphs $G \times H$ with the parameter matrix $S \otimes T$.
\end{teorema}

\begin{teorema} \label{structdecprod}
The matrix  $ P \otimes R$ is a perfect structure in the Cartesian product of graphs $G \Box H$ with the parameter matrix $ I \otimes T + S \otimes I$.
\end{teorema}

\begin{teorema} \label{structnormalprod}
The matrix  $P \otimes R$ is a perfect structure in the normal product of graphs $G \boxtimes H$ with the parameter matrix $ I \otimes T + S \otimes I + S \otimes T$.
\end{teorema}

To apply Theorem~\ref{structgenprod} to the lexicographic product of graphs, we need an additional condition on graphs $G$ and $H$. 

\begin{teorema} \label{structlexprod}
If $(M, P, S)$ and $(L, R, T)$ are perfect structures in graphs $G$ and $H$ with adjacency matrices $M$ and $L$ and if $(J, R, T')$  is  a perfect structure with some parameter matrix $T'$, then the matrix $P \otimes R$ is a perfect structure in the lexicographic product of graphs $G \cdot H$ with the parameter matrix $S \otimes T' + I \otimes T$.
\end{teorema}

Another application of Theorem~\ref{structgenprod} is calculating the spectra of graphs.

\begin{teorema} \label{specprodgraph}
If for a collection of matrices $M_1, \ldots, M_{m}$  (and matrices $L_1$, $\ldots$, $L_{l}$) of order $n'$ (of order $n''$)  there exists a consolidated full set of eigenvectors, then the spectrum of the $(\alpha_{i,j})$-product of matrices $M_1, \ldots, M_{m}$ and $L_1, \ldots, L_{l}$ is equal to
$$sp \left( \sum\limits_{i=1}^m \sum\limits_{j=1}^l \alpha_{i,j} M_{i} \otimes L_{j}\right) =\left\{\sum\limits_{i=1}^m \sum\limits_{j=1}^l \alpha_{i,j} \mu^{i}_s \lambda^j_t |  s = 1, \ldots, n', t = 1, \ldots, n''  \right\} $$
as a multiset, where $\mu^i_s$ and $\lambda^j_t$ are eigenvalues of matrices $M_i$ and $L_j$ respectively.
\end{teorema}
\begin{proof}
Theorem~\ref{structgenprod} implies that the tensor product of each pair of eigenvectors of matrices $M_i$ and $L_j$ is an eigenvector in the $(\alpha_{i,j})$-product of matrices corresponding to the demanded eigenvalue. The number of eigenvectors is equal to $n'n''$, so the construction gives a full set of eigenvalues and all linearly independent eigenvectors of the $(\alpha_{i,j})$-product.
\end{proof}

Using this theorem and the fact that every vector is an eigenvector for the identity matrix $I$, we find spectra of the tensor, Cartesian and normal products of graphs. In the following theorems we assume that $f$ is an eigenfunction in a graph $G$ corresponding to an eigenvalue $\mu$ and $g$ is an eigenfunction in a graph $H$ corresponding to an eigenvalue $\lambda$.

\begin{teorema} \label{sobftensor}
Function $f \otimes g$  is an eigenfunction in the tensor product  of graphs $G \times H$ corresponding to the eigenvalue $\mu \lambda$. The spectrum of the tensor product is
$$sp(G \times H) = \{ \mu_i \lambda_j |  \mu_i \in sp(G), \lambda_j \in  sp(H) \}.$$
\end{teorema}

\begin{teorema} \label{sobfdecart}
Function $f \otimes g$  is an eigenfunction in the Cartesian product of graphs $G \Box H$ corresponding to the eigenvalue $\mu + \lambda$. The spectrum of the Cartesian product is
$$sp(G \Box H) = \{ \mu_i + \lambda_j | \mu_i \in sp(G), \lambda_j \in  sp(H) \}.$$
\end{teorema}

\begin{teorema} \label{sobfnorm}
Function $f \otimes g$  is an eigenfunction in the normal product of graphs  $G \boxtimes H$ corresponding to the eigenvalue $\mu + \lambda + \mu \lambda$. The spectrum of the normal product is 
$$sp(G \boxtimes H) = \left\{\mu_i + \lambda_j + \mu_i\lambda_j | \mu_i \in sp(G), \lambda_j \in  sp(H)\right\}.$$
\end{teorema}

In the case of the lexicographic product of $G$ and $H$, the application of Theorem~\ref{specprodgraph} requires that the adjacency matrix of the graph  $H$ and the unity matrix $J$ have a common full set of eigenvectors. By Corollary~\ref{commonJ}, to satisfy this condition it is sufficient to require the regularity of graph $H$.

For an $r$-regular graph $H$ on $n$ vertices and for an eigenvalue $\lambda$ of this graph, let  $\tilde{\lambda}$  denote the  eigenvalue  of the unity matrix $J$ corresponding to the same eigenvector. Note that $\tilde{\lambda} = n$ if $\lambda = r$ and $\tilde{\lambda} = 0$ otherwise.

\begin{teorema} \label{sobflexprod}
If $f$ and $g$ are eigenfunctions in a graph $G$ and in a regular graph $H$ corresponding to eigenvalues $\mu$ and $\lambda$ respectively, then $f \otimes g$  is an eigenfunction in the lexicographic product $G \cdot H$ corresponding to the eigenvalue $\mu \tilde{\lambda} + \lambda$. The spectrum of the lexicographic product is 
$$sp(G \cdot H) = \left\{\mu_i \tilde{\lambda_j} + \lambda_j | \mu_i \in sp(G), \lambda_j \in sp(H)\right\}.$$
\end{teorema}

In conclusion, we note that Theorem~\ref{structgenprod} for perfect structures and Theorem~\ref{specprodgraph} on the graph spectra can be specialized for many other cases of $(\alpha_{i,j})$-product,  including, for example,  conormal or disjunctive product, modular product, double graph, etc. 

\subsection{The contraction of eigenfunctions in graph products} \label{projectpar}

The main result of this subsection is an inversion of Theorem~\ref{specprodgraph}  on spectrum and eigenfunctions in graph products with respect to one of the graphs in the product.

To avoid cumbersome notations but still consider a quite general case, we state the contraction theorem for $(\alpha_{i,j})$-products being a sum of two tensor products of matrices. For $(\alpha_{i,j})$-products consisting of more summands, the theorem and its proof are similar.

Let $(M_1, M_2)$ and $(L_1, L_2)$ be pairs of diagonalizable matrices such that the matrices within a pair have orders $m$ and $n$ respectively. Assume that for pairs $(M_1, M_2)$ and $(L_1, L_2)$ there exist  full sets $\{ f_i \}_{i=1}^{m}$  and $\{ g_j\}_{j=1}^{n}$ of pairwise orthogonal eigenvectors:
$$M_1 f_i = \mu'_i f_i;~~~ M_2 f_i = \mu''_i f_i; ~~~ L_1 g_j = \lambda'_j g_j; ~~~ L_2 g_j = \lambda''_j g_j.$$

Define the product of these matrices to be the matrix  $N = M_1 \otimes L_1 + M_2 \otimes L_2$ of order $mn$. Next, a vector $h$ of length $mn$ can be considered as a member of the tensor product of spaces $\mathbb{C}^m$ and $\mathbb{C}^n$, so we can write $h = \sum\limits_{i=1}^{m}\sum\limits_{j=1}^{n} h(i,j) e^m_i \otimes e^n_j$, where $e^m_i$ and $e^n_j$ are the standard basis vectors in these spaces. Given a function $g$ from the space $\mathbb{C}^n$, put $P^h_g = Hg$, where $H$ is the  $m \times n$ matrix with entries $h(i,j)$.

\begin{teorema}[Contraction theorem] \label{phiproekt}
Under the above conditions, assume that  the triples $(N, h, \nu)$, $(L_1, g, \lambda')$ and  $(L_2, g, \lambda'')$ are eigenfunctions and that the eigenvalue $\lambda''$ is not equal to zero. If the function $f = P^h_g$ is an eigenfunction  $(M_1,f, \mu')$, then the triple  $(M_2, f, \mu'')$ is also an eigenfunction with the eigenvalue $\mu'' = (\nu - \lambda' \mu')/\lambda''$.
\end{teorema}

\begin{proof}
Without loss of generality, we assume that eigenfunctions $f = f_1$ and $g = g_1$. Then eigenvalues $\mu' = \mu'_1$, $\lambda' = \lambda'_1$, and $\lambda'' = \lambda''_1$. Since $\{ f_i \}_{i=1}^{m}$  and $\{ g_j\}_{j=1}^{n}$ are the full sets of eigenvectors for the given matrices, Theorem~\ref{structgenprod} implies that the vectors $ f_i \otimes g_j $ form the full set of eigenvectors of the matrix $N$ and correspond to the eigenvalues $\mu'_i \lambda'_j + \mu''_i \lambda''_j $.

 Let $\mathcal{I}$ denote the set of all pairs $(i,j)$ such that $\mu'_i \lambda'_j + \mu''_i \lambda''_j  = \nu$. Since $h$ is an eigenvector of the matrix $N$ corresponding to the eigenvalue $\nu$, we have
$$h = \sum\limits_{(i,j)\in \mathcal{I}} \beta_{i,j} f_i \otimes g_j  $$
for some complex $\beta_{i,j}$.
Using the equivalence between the Knocker product $f_i \otimes g_j$ and the outer product $f_i g_j^T$ and the  decomposition of $h$ with respect to the standard basis, we write the matrix $H$ as $H =  \sum\limits_{(i,j)\in \mathcal{I}} \beta_{i,j} f_i  g_j^T. $

Because all vectors $\{ g_j\}_{j=1}^n$ are orthogonal, for $f = Hg_1$ we have
$$f = \left(  \sum\limits_{(i,j)\in \mathcal{I}} \beta_{i,j} f_i  g_j^T \right) g_1 =   \sum\limits_{(i,j)\in \mathcal{I}} \beta_{i,j} f_i  (g_j^T  g_1) = \sum\limits_{(i,1)\in \mathcal{I}} ||g_1||^2  \beta_{i,1}  f_i.$$
By the condition, the vector $f = f_1$ is an eigenvector of the matrix $M_1$ corresponding to the eigenvalue $\mu' = \mu'_1$, consequently all vectors $f_i$ in the last sum are eigenvectors of the matrix $M_1$ corresponding to the same eigenvalue $\mu' = \mu'_i$. 
Using the definition of the set $\mathcal{I}$ and equalities $\lambda' = \lambda'_1$ and $\lambda'' = \lambda''_1$, we obtain that the vector $f$ is a linear combination of eigenvectors of $M_2$ corresponding to the eigenvalue $\mu'' = (\nu - \lambda'\mu')/ \lambda''$. 
\end{proof}

We show the essence of Theorem~\ref{phiproekt}  on examples of the basic graph products.

Let $M$ and $L$ be the adjacency matrices of simple graphs $G = (V,E)$ and $H = (U,W)$. Since the matrices $M$ and $L$ are symmetric, each of them has a full set of pairwise orthogonal eigenvectors. Because the other matrices used in the tensor, Cartesian and normal products are the identity matrices $I$ of an appropriate order, Proposition~\ref{structI} guarantees that there exists the required common full set of pairwise orthogonal eigenvectors for collections of matrices. 

Given a function $h = h(u,v)$ defined on the product $V \times U$ and the function $g = g(u)$ on the vertex set $U$ of the graph $H$, we define the function $P^h_g(v) = \sum\limits_{u \in U} h(v,u) g(u)$ on the vertex set $V$ of $G$.

\begin{teorema} \label{phiproekttensor}
If  $(M \otimes L, h, \nu )$ and $(L, g, \lambda)$  are eigenfunctions in the tensor product $G \times H$ and in the graph $H$ respectively and the eigenvalue $\lambda$ is not equal to zero, then $(M, f, \mu)$ is an eigenfunction in the graph $G$, where $f = P^h_g$ and $\mu= \nu / \lambda$.
\end{teorema}

\begin{teorema} \label{phiproektdec}
If  $(M \otimes I + I \otimes L, h, \nu )$ and $(L, g, \lambda)$  are eigenfunctions in the Cartesian product $G \Box H$ and in the graph $H$ respectively, then $(M, f, \mu)$ is an eigenfunction in the graph $G$, where  $f = P^h_g$ and $\mu = \nu - \lambda$.
\end{teorema}

\begin{teorema} \label{phiproektnorm}
If  $(M \otimes I + I \otimes L + M \otimes L, h, \nu )$ and $(L, g, \lambda)$ are eigenfunctions in the normal product $G \boxtimes H$ and in the graph $H$ and the eigenvalue $\lambda$ is not equal to  $-1$, then $(M, f, \mu)$ is an eigenfunction in the graph $G$,  where  $f = P^h_g$ and $\mu = \frac{\nu - \lambda}{1 + \lambda}$.
\end{teorema}

To establish Theorem~\ref{phiproekt} for the lexicographic product, Corollary~\ref{commonJ} allows us to require just the $r$-regularity of the graph $H$ as an additional condition. Note that  in this case Theorem~\ref{phiproekt} can be stated only for the largest eigenvalue $\lambda = r$ of the graph $H$.

\begin{teorema} \label{phiproektlex}
If  $(M \otimes J + I \otimes L, h, \nu )$ and $(L, \mathbb{1}, r)$ are eigenfunctions in the lexicographic products $G \cdot H$ and in the $r$-regular graph $H$, then $(M, f, \mu)$ is an eigenfunction in the graph $G$,  where  $f(v) = P^h_{\mathbb{1}} =  \sum\limits_{u \in U} h(v,u)$ and $\mu = \frac{\nu - r}{|U|}$.
\end{teorema}

\section{Examples and applications} \label{primsec}

\subsection{Calculating spectra of some graphs} \label{spectrsection}

In this subsection we use Theorem~\ref{specprodgraph} and its corollaries (Theorems~\ref{sobftensor} -- \ref{sobflexprod}) for finding spectra of a series of graphs.

Let us recall that the spectrum of a graph $G$ is a multiset  $\{ \mu_1^{[\nu_1]}, \ldots,\mu_s^{[\nu_s]}\}$ of eigenvalues of its adjacency matrix $M(G)$, where $\nu_1, \ldots, \nu_s$ denote the multiplicities of the eigenvalues.

We start with the spectrum of the complete graph.

\begin{utv}
The spectrum of the complete graph $K_n$ on $n$ vertices is
$$sp(K_n) = \{ (-1)^{[n-1]}, (n-1)^{[1]} \}.$$
\end{utv}
 
A \textit{matching} $M_n$ is a graph on $2n$ vertices and with $n$ edges such that each vertex is incident to exactly one edge.

\begin{utv}
 The spectrum of the matching graph $M_n$ on $2n$  vertices is
$$sp(M_n) =  \left\{ (-1)^{[n]}, 1^{[n]} \right\}.$$
\end{utv}

\begin{proof}
The matching graph is the tensor product of the graph with adjacency matrix $I$ and the single-edge graph. 
It remains to note that $sp(I) = \{ 1^{[n]}\}$ and the spectrum of the single-edge graph is $\{ (-1)^{[1]}, 1^{[1]}\}$.
\end{proof}

The \textit{complete bipartite graph} $K_{n,n}$ is a graph on $2n$ vertices divided into two parts $U$ and $V$ of the same size $n$, edges of $K_{n,n}$ are pairs $uv$, where $u \in U$ and $v \in V$.

\begin{utv}
The spectrum of the complete bipartite graph $K_{n,n}$ with parts of size $n$ is
$$sp(K_{n,n}) =  \left\{ (- n)^{[1]}, 0^{[2n-2]}, n^{[1]} \right\}.$$
\end{utv}

\begin{proof}
The complete bipartite graph $K_{n,n}$ is the tensor product of the single-edge graph and the $n$-vertex graph with the adjacency matrix $J$.
It remains to note that  $sp(J) = \{ 0^{[n-1]}, n^{[1]}\}$ and the spectrum of the single-edge graph is $\{ (-1)^{[1]}, 1^{[1]}\}$.
\end{proof}

Similarly, the \textit{complete $k$-partite graph} $K_{n , \ldots, n}$ is the graph on $kn$ vertices divided into $k$ parts of size $n$ such that $uv$ is an edge of the graph $K_{n,\ldots, n}$  if and only if vertices $u$ and $v$ belong to different parts.

\begin{utv}
The spectrum of the complete $k$-partite graph $K_{n, \ldots, n}$ with parts of size $n$ is
$$sp(K_{n, \ldots, n}) =  \left\{(-n)^{[k-1]}, 0^{[k(n-1)]}, (n(k - 1))^{[1]} \right\}.$$
\end{utv}

\begin{proof}
The complete $k$-partite graph $K_{n, \ldots, n}$ is the tensor product of the complete graph $K_n$ and the $n$-vertex graph with the adjacency matrix $J$.
It remains to note that $sp(K_k) = \{ (-1)^{[k-1]}, (k-1)^{[1]}\}$ and that  $sp(J) = \{ 0^{[n-1]}, n^{[1]}\}$.
\end{proof}

The \textit{Hamming graph} $H(n,q)$ is a graph on $q^n$ vertices that is the $n$-th  Cartesian power of the complete graph $K_q$.

\begin{utv}
The spectrum of the Hamming graph $H(n,q)$ is 
$$sp(H(n,q)) = \{  (-n)^{[(q-1)^n]}, \ldots,  (n(q-1) - qi)^{\left[{n \choose i} (q-1)^i\right]}, \ldots, (n(q-1))^{[1]}  \}.$$
\end{utv}
\begin{proof}
The statement is easy to prove by induction on $n$ with the help of Theorem~\ref{sobfdecart} and the fact that $sp(K_q) = \{ (-1)^{[q-1]}, (q-1)^{[1]}\}$.
\end{proof}

A \textit{path} $P_n$ is a graph on $n$ vertices $v_1, \ldots, v_n$ with edges $v_{i}v_{i+1}$, $i =1, \ldots, n-1$. A \textit{cycle}  $C_n$ is a graph on $n$ vertices  obtained from the path $P_n$ by adding the  edge $v_1 v_n$. The following statement on the spectra of path and cycle graphs can be found, for example, in book~\cite{doob}.

\begin{utv} \label{sppathcyc}
The spectrum of the path graph $P_n$ is
$$sp(P_n) = \left\{  \left( 2\cos \frac{\pi i }{n+1} \right)^{[1]} ~ \vline ~  i = 1, \ldots, n   \right\}.$$
The spectrum of the cycle $C_n$ is
$$sp(C_n) = \left\{  \left( 2\cos \frac{2 \pi i }{n} \right)^{[1]} ~ \vline ~ i = 1, \ldots, n \right\}.$$
\end{utv}

The \textit{grid} graph $Gr_{m,n}$ is a graph on $m n$ vertices that is the Cartesian product of paths $P_m$ and $P_n$. The \textit{ladder} graph $L_n$ is a graph on $2n$ vertices obtained as the Cartesian product of the path $2n$ and  the single-edge graph (a path $P_2$).

\begin{utv}
The spectrum of the grid graph $Gr_{m,n}$ is
$$sp(Gr_{m,n}) = \left\{  \left( 2\cos \frac{\pi i }{m+1} + 2\cos \frac{\pi j}{n+1} \right)^{[1]} ~ \vline ~  i = 1, \ldots, m; ~ j = 1, \ldots, n   \right\}.$$
In particular, the spectrum of the ladder  $L_n$ is
$$sp(L_{n}) = \left\{  \left( 2 \cos \frac{\pi i }{n+1} -1 \right)^{[1]},  ~ \left( 2  \cos \frac{\pi i }{n+1} +1 \right)^{[1]} ~ \vline ~  i = 1, \ldots, n;  \right\}$$
\end{utv}

\begin{proof}
It is sufficient to use Theorem~\ref{sobfdecart} and the spectrum of path $P_n$ from Proposition~\ref{sppathcyc} to prove the proposition.
\end{proof}

A \textit{square grid on torus} $Tr_{m,n}$ is a graph on $mn$ vertices that is the Cartesian product of cycles $C_m$ and $C_n$.
A \textit{prism} $Pr_n$ is a graph on $2n$ vertices obtained as the Cartesian product of the cycle $C_n$ and the single-edge graph.

\begin{utv}
The spectrum of the square grid on torus $Tr_{m,n}$ is
$$sp(Tr_{m,n}) = \left\{   \left( 2\cos \frac{2\pi i }{m} + 2\cos \frac{2\pi j}{n} \right)^{[1]} ~ \vline ~  i = 1, \ldots, m; ~ j = 1, \ldots, n   \right\}.$$
In particular, the spectrum of the prism graph $Pr_{n}$ is
$$sp(Pr_{n}) = \left\{  \left( 2  \cos \frac{2\pi i }{n} -1 \right)^{[1]}, ~\left(2 \cos \frac{2\pi i }{n} +1 \right)^{[1]}  ~ \vline ~ i = 1, \ldots, n   \right\}.$$
\end{utv}

\begin{proof}
It is sufficient to use Theorem~\ref{sobfdecart} and the spectrum of the path $C_n$ from Proposition~\ref{sppathcyc} to prove the proposition.
\end{proof}

In conclusion, we consider two graph operations that can be presented as a graph product.

The \textit{double graph} $D(G)$ of  a graph $G$ is a graph obtained by taking two copies of $G$ (including all its edges) and adding all edges $u_1 v_2$ between the copies for each edge $uv$ of the graph $G$. The double graph $D(G)$ is the tensor product of the graph $G$ and the graph with adjacency matrix $J$ on $2$ vertices. By Theorem~\ref{sobftensor}, we have the following statement.

\begin{utv}
Given a graph $G$ and its spectrum $sp(G)$, the spectrum of the double graph $D(G)$ is
$$sp(D(G)) = \{ 0^{[n]} \} \cup( 2 \cdot sp(G)).$$
\end{utv}

The \textit{bipartite double graph} $BD(G)$ of a graph $G$ is a graph obtained by taking two copies of the vertex set of $G$ and adding all edges $u_1 v_2$ between the copies for each edge $uv$ of $G$. The bipartite double graph is the tensor product of the graph $G$ and the single-edge graph. Theorem~\ref{sobftensor} implies the following.

\begin{utv}
Given a graph $G$ and its spectrum $sp(G)$, the spectrum of the bipartite double graph $BD(G)$ is
$$sp(BD(G)) = sp(G) \cup (-sp(G)).$$
\end{utv}

\subsection{Applications to perfect colorings} \label{perfcolpar}

The constructions of perfect structures in graph products give a series of perfect colorings and their parameter matrices for many graph classes. So it is a handy and important tool in the characterization of perfect colorings for a given graph.  To simplify the application of these constructions in future works, here we state them specially for perfect colorings.

Let $(M, P, S)$ and $(L, R, T)$ be perfect colorings in $k_1$ and $k_2$ colors in graphs $G$ and $H$ having adjacency matrices $M$ and $L$ and parameter matrices $S$ and $T$, respectively. Theorems~\ref{structtensprod} -- \ref{structnormalprod} imply the following.

\begin{teorema} \label{colortensorprod}
The matrix $P \otimes R$ is a perfect coloring in $k_1 k_2$ colors of the tensor product of graphs $G \times H$ with the parameter matrix $S \otimes T$.
\end{teorema}

\begin{teorema} \label{colordecprod}
The matrix $P \otimes R$ is a perfect coloring in $k_1k_2$ colors of the Cartesian product of graphs $G \Box H$ with the parameter matrix $I \otimes T + S \otimes I$.
\end{teorema}

In the Cartesian product of graphs, a perfect coloring obtained as the product of some coloring and the trivial coloring in one color is often called reducible.

\begin{teorema} \label{colornormalprod}
The matrix $P \otimes R$  is a perfect coloring in $k_1k_2$ colors of the normal product of graphs $G \boxtimes H$ with the parameter matrix $ S \otimes I + I \otimes T + S \otimes T$.
\end{teorema}

To state a similar result for the lexicographic product, we use Theorem~\ref{structlexprod} and Corollary~\ref{colorinJ}.

\begin{teorema} \label{colorlexprod}
The matrix $P \otimes R$ is a perfect coloring in $k_1 k_2$ colors of the lexicographic product of graphs $G \cdot H $ with the parameter matrix $S \otimes (J \cdot \diag(l_1, \dots, l_{k_2})) + I \otimes T$, where $l_1, \ldots, l_{k_2}$ are the numbers of vertices of each color in the perfect coloring $(L,R,T)$.
\end{teorema}

In conclusion, we state one more result that is known as the orthogonality theorem for perfect colorings.

\begin{teorema}[Orthogonality of perfect colorings]
Let $(M,P,S)$ and $(M,R,T)$ be perfect colrings of a connected $r$-regular graph on $n$ vertices with the adjacency matrix $M$. Assume that the spectra of the parameter matrices $S$ and $T$ have in common only the eigenvalue $r$. Then for each pair of columns $P_i$ and $R_j$ of the structure matrices $P$ and $R$ their dot product 
$$\langle P_i, R_j \rangle = \frac{l_i m_j}{n},$$
where $l_i$ is the number  of vertices of color $i$ in the coloring $P$, and  $m_j$  is the number of vertices of color $j$ in the coloring $R$.
\end{teorema}

In other words, this theorem states that if we narrow a perfect coloring $P$ on a color class of some other coloring $R$, such that the spectra of parameter matrices of these two colorings intersect only at the degree of the graph, then the density $l_i / n$ of color $i$ of the perfect coloring $P$ does not change.

\begin{proof}
By the definition, we have  $l_i = \langle P_i , \mathbb{1} \rangle$ and $m_j = \langle R_j,\mathbb{1} \rangle$. 

Since all eigenvectors of $M$ corresponding to the eigenvalue $r$ are collinear to the vector $\mathbb{1}$, Property~\ref{columnspaceP} implies that the vector $f = P_i - \frac{l_i}{n} \mathbb{1}$ belongs to the sum of eigenspaces of the matrix $M$ corresponding the eigenvalues $sp(S) \setminus \{ d\}$. Similarly, the vector   $g = R_j - \frac{m_j}{n} \mathbb{1}$ belongs to the sum of eigenspaces corresponding to the eigenvalues $sp(T) \setminus \{ d\}$.

By the conditions, the sets  $sp(S) \setminus \{ d\}$ and $sp(T) \setminus \{ d\}$ are disjoint and the matrix $M$ is diagonalizable. It follows that the eigenspaces corresponding to these eigenvalues are orthogonal, and, consequently, the vectors $f$ and $g$ are also ortogonal. Thus
$$0 = \langle f, g \rangle = \left\langle P_i - \frac{l_i}{n} \mathbb{1}, R_j - \frac{m_j}{n} \mathbb{1} \right\rangle = \langle P_i, R_j \rangle - \frac{l_i m_j}{n},$$
that gives the statement of the theorem.
\end{proof}

\section*{Acknowledgements}

The author is grateful to S.V. Avgustinovich and V.N. Potapov for useful discussion and constant attention to this work. The work was funded by the Russian Science Foundation under grant 18--11--00136 (Sections 2 -- 4) and is supported in part by the Young Russian Mathematics award (Section 1).

\end{document}